\documentclass[12pt]{article}

\usepackage{amsmath,amsthm,amsfonts,amssymb,amscd,cite}
\usepackage{graphicx}

\allowdisplaybreaks[4]

\newtheorem {lemma}{Lemma}
\newtheorem {theorem} {Theorem}
\newtheorem {claim}{Claim}
\allowdisplaybreaks [4]
\usepackage{fullpage}

\begin{document}

\title{Extremal $Q$-index problem in outerplanar graphs}

\author{Jin Cai\footnote{Email: jincai@m.scnu.edu.cn}, Leyou Xu\footnote{Email: leyouxu@m.scnu.edu.cn}, Bo Zhou\footnote{Email: zhoubo@scnu.edu.cn}\\
School of Mathematical Sciences, South China Normal University\\
Guangzhou 510631, P.R. China}

\date{}
\maketitle

\begin{abstract}
Outerplanar Tur\'an problem has received considerable attention recently. We study the spectral version via $Q$-index. We determine the unique graph that maximizes the $Q$-index among all $n$-vertex connected outerplanar graphs which are respectively forbidden to contain: (i) a fixed cycle; and (ii) the disjoint union of paths of a given order.\\ \\ 
{\bf Keywords:}
$Q$-index, outerplanar graph, outerplanar Tur\'an problem\\ \\
{AMS classification:} 05C50, 05C10
\end{abstract}

\section{Introduction}

For two vertex disjoint graphs $G$ and $H$, $G\cup H$ denotes the union of $G$ and $H$, and $G\vee H$ denotes the join of $G$ and $H$. For a positive integer $a$, $aG$ denotes the union of $a$  (vertex disjoint) copies of a graph $G$.   A complete graph on $n$ vertices is denoted by $K_n$.  A path  on $n$ vertices is denoted by $P_n$.  A cycle  on $n$ vertices is denoted by $C_n$, where $n\ge 3$.

Given a graph $F$, a graph $G$ is said to be $F$-free if $G$ does not contain $F$ (as a subgraph). The well known Tur\'an-type problem is to determine the maximum number of edges of an $n$-vertex $F$-free graph (as well as the $n$-vertex $F$-free graph(s) possessing exactly this number of edges). 
The classic Tur\'an theorem \cite{Tur} states that the balanced complete $r$-partite graph on $n$ vertices is the unique extremal graph maximizing the number of edges among all $n$-vertex $K_{r+1}$-free graphs. The research on Tur\'an-type problems is well summarized in the survey \cite{FS}.

In 2016, Dowden \cite{Dow} initiated the study of the Tur\'an-type problem for planar graphs. Given a planar graph $F$, this problem asks for  the maximum number of edges of an $n$-vertex $F$-free planar graph. Among others, there are many results for $F$ being cycles or paths, see the  survey \cite{LSS} (and reference therein) and subsequent papers, e.g., \cite{CLL,GGMPX,GVZ,SWY,SWY2}.

A graph $G$ is outerplanar if it has a planar embedding in which all vertices lie on the boundary of its outer face. Given an outerplanar graph $F$, one may want to know
the maximum number of edges of an $n$-vertex $F$-free outerplanar graph.  
Recently, Matolcsi and Nagy \cite{MN} determined this number for $F=P_3$, Gy\"ori  et al \cite{GPX} determined this number for $F=P_4, P_5$. In \cite{FZ}, Fang and Zhai determined this  number  for $F=C_\ell$ with $\ell\ge 3$ or $F=P_\ell$ with $\ell\ge 4$. 

Let $A(G)$ denote the adjacency matrix of a graph $G$. The spectral radius of $G$ is the largest eigenvalue of $A(G)$. In 1990, Cvetkovi\'c and Rowlinson \cite{CR} proposed  a conjecture on the maximum spectral radius of outerplanar graphs. Tait and Tobin \cite{TT} confirmed  the conjecture for graphs of sufficiently large order in 2017, and Lin and Ning \cite{LN} completely confirmed this conjecture in 2021. Recently, Sun et al. \cite{SWS} determined the maximum spectral radius of $n$-vertex $C_\ell$-free outerplanar graphs, and Yin and Li \cite{YL} considered the maximum spectral radius of $n$-vertex $F$-free outerplanar graphs when $F$ is either the graph consisting of  $t$-edge disjoint cycles $C_\ell$ with a common vertex or $F=(t+1)K_2$. 

The signless Laplacian matrix of a graph $G$ is the matrix $Q(G)=D(G)+A(G)$, where $D(G)$ is the degree diagonal matrix of $G$. The largest eigenvalue of $Q(G)$ is known as the $Q$-index of $G$, denoted by $q(G)$. Yu et al. \cite{YWG} considered the maximum $Q$-index among all $n$-vertex planar graphs. For outerplanar graphs, Yu et al. \cite{YGW} showed that $K_1\vee P_{n-1}$  is the unique graph that maximizes the $Q$-index.    
The study on the $Q$-index version of Tur\'an-type problem has attracted much attention, see, e.g.,  \cite{CJZ,CLZ,NY,ZW}.

Motivated by the above results, we focus on the $Q$-index version of Tur\'an-type problem on outerplanar graphs in this paper.

Since $K_1 \vee P_{n-1}$ maximizes the $Q$-index among all $n$-vertex outerplanar graphs, and $K_1 \vee P_{n-1}$ does not contain two disjoint cycles, it is natural to restrict our attention to graphs forbidding a single cycle of fixed length. Therefore, in the context of cycles, we study the $Q$-index extremal problem for outerplanar graphs that do not contain a given cycle $C_\ell$.

\begin{theorem}\label{cycle}
Let $G$ be a connected outerplanar $C_\ell$-free graph of order $n$ with $\ell\ge 3$. Then $q(G)\le q(K_1\vee (\alpha P_{\ell-2}\cup P_{r}))$ with equality if and only if $G\cong K_1\vee (\alpha P_{\ell-2}\cup P_{r})$, where $\alpha =\lfloor\frac{n-1}{\ell-2}\rfloor$ and $r=n-1-\alpha(\ell-2)$.
\end{theorem}

In the context of paths, we consider forbidding a single path and forbidding the disjoint union of $t$ paths of the same length. Interestingly, the extremal graphs for these two types of problems behave somewhat differently. Since any connected graph of order $n\ge 3$ contains $P_2$ and $P_3$, for the case of forbidding a single path, we restrict our attention to paths of order at least four.

\begin{theorem}\label{path}
Let $G$ be a connected outerplanar $tP_\ell$-free graph of order $n$ with $t\ge 1$. 

(i) If $t=1$, $\ell\ge 4$ and $n\ge \max\Big\{\lfloor\frac{\ell-3}{2} \rfloor^2+\lfloor\frac{\ell-3}{2} \rfloor+\ell-1,30\lfloor\frac{\ell-3}{2} \rfloor+30\sqrt{\lfloor\frac{\ell-3}{2} \rfloor\big(\lfloor\frac{\ell-3}{2} \rfloor+1\big)}\Big \}$, then $q(G)\le q\big (K_1\vee \big(P_{\lceil\frac{\ell-2}{2} \rceil}\cup\alpha P_{\lfloor\frac{\ell-2}{2} \rfloor}\cup P_r\big)\big)$ with equality if and only if $G\cong K_1\vee \big(P_{\lceil\frac{\ell-2}{2} \rceil}\cup\alpha P_{\lfloor\frac{\ell-2}{2} \rfloor}\cup P_r\big)$, where $\alpha=\left \lfloor \frac{n-\ell+1}{\lfloor\frac{\ell-2}{2} \rfloor}\right\rfloor+1$ and $r=n-\ell+1-(\alpha-1)\lfloor\frac{\ell-2}{2} \rfloor$.

(ii) If $t\ge 2$, $\ell\ge 2$ and $n\ge \max\Big\{\ell^2+(t-3)\ell+1,30(\ell-2)+30\sqrt{(\ell-2)(\ell-1)}\Big\}$, then $q(G)\le q(K_1\vee (P_{t\ell-\ell-1}\cup\alpha P_{\ell-1}\cup P_r))$ with equality if and only if $G\cong K_1\vee (P_{t\ell-\ell-1}\cup\alpha P_{\ell-1}\cup P_r)$, where $\alpha=\lfloor \frac{n-t+2}{\ell-1}\rfloor-t+1$ and $r=n-2\ell+2-(\alpha-1)(\ell-1)$.
\end{theorem}


To prove the above results, we  establish a structural theorem for the extremal graphs. 

\begin{theorem}\label{str}
Let $F=C_\ell$ with $3\le \ell\le n$ or $F=tP_k$ with $k\ge 2$ and $4\le tk\le n-1$. If $G$ maximizes the $Q$-index among all connected $n$-vertex $F$-free outerplanar graphs, then $G$ contains a vertex $u$ of degree $n-1$ and $G[N(u)]$ consists of paths. 
\end{theorem}

The paper is arranged as follows. In Section 2, we introduce the basic notations and preliminary lemmas that will be used throughout the paper. In Section 3, we prove Theorem \ref{str}. In Section 4, we prove Theorems \ref{cycle} and \ref{path}.

\section{Preliminaries}

All graphs considered here are simple and undirected. Let $G$ be a graph with vertex set $V(G)$ and edge set $E(G)$.
For any $v\in V(G)$, we denote by $N_G(v)$ the set of vertices adjacent to $v$ in $G$.
Let $N_G[v]=N_G(v)\cup \{v\}$ be the closed neighborhood of $v$ and $d_G(v)=|N_G(v)|$ the degree of $v$. We write $N(v)$ and $d(v)$ for $N_G(v)$ and $d_G(v)$ if there is no ambiguity. For $S\subseteq V(G)$, let $N_S(v)=N_G(v)\cap S$ and $d_S(v)=|N_{S}(v)|$.
For $S,T\subseteq V(G)$ with $S\cap T=\emptyset$, let $e(S,T)$ be the number of edges between $S$ and $T$, and $G[S]$ the subgraph of $G$ induced by $S$. We write $e(S)$ for $e(G[S])$.
For $S\subset E(G)$, let $G-S$ denote the graph obtained from $G$ by deleting edges in $S$. Particularly, if $S=\{e\}$, we write $G-e$ for $G-\{e\}$. Conversely, if $e\notin E(G)$, $G+e$ denotes the graph obtained from $G$ by adding an edge $e$. 

Let $K_{a,b}$ denote the complete bipartite graph with partite size $a$ and $b$.

A graph $H$ is a minor of a graph $G$ if $H$ can be obtained from a subgraph of $G$ by contracting edges. A graph $G$ is called $H$-minor free if $G$ does not contain $H$ as a minor.
It is well-known that an outerplanar graph is $K_{2,3}$-minor and $K_4$-minor free.
Based on this fact, we have the following observations. 

\begin{lemma}\label{obv}\cite{Bon}
Let $G$ be an outerplanar graph $G$ of order $n$. Then $e(G)\le 2n-3$. Moreover, for any  $u\in V(G)$, the following statements hold. \\
(i) $G[N(u)]$ consists of paths; \\
(ii) $|N(u)\cap N(v)|\le 2$ for any $v\in V(G)\setminus \{u\}$; \\
(iii) If $|N(u)\cap N(v)|=2$ for some $v\in V(G)\setminus N_G[u]$, then the two neighbors of $v$ are either adjacent or in different components of $G[N(u)]$. Moreover, if the two neighbors of $v$ lie in different components of $G[N(u)]$, then both of them are pendant vertices in $G[N(u)]$; \\
(iv) If $|N(u)\cap N(v_1)|=|N(u)\cap N(v_2)|=2$ for some $v_1,v_2\in V(G)\setminus N_G[u]$, then $(N(v_1)\cap N(u))\cap (N(v_2)\cap N(u))=\emptyset$.
\end{lemma}

For a square matrix $M$, denote by $\lambda(M)$ the largest eigenvalue of $M$.
By \cite[Corollary 2.2, p.~38]{Mi}, for nonnegative matrices $A$ and $C$  of order $n$,
if $A-C$ is nonnegative, nonzero and $A$ is irreducible, then $\lambda(A)>\lambda(C)$.
Applying this result to $Q(G)$ for a connected graph $G$ that is not complete, we have the following lemma.

\begin{lemma}\label{addedges}
Let $G$ be a connected graph  with $u,v\in V(G)$ and $uv\notin E(G)$.
Then $q(G+uv)>q(G)$.
\end{lemma}

\begin{lemma}\label{Delta}\cite{CS}
Let $G$ be a graph with maximum degree $\Delta(G)\ge 1$. Then $q(G)\ge \Delta(G)+1$ with equality if and only if $G\cong K_{1,n-1}$.
\end{lemma}

For vertex $u\in V(G)$, let \[
\eta_G(u)=d_G(u)+\frac{1}{d_G(u)}\sum_{v\in N_G(u)}d_G(v).
\]
We write $\eta(u)$ for $\eta_G(u)$ if there is no ambiguity.

\begin{lemma}\label{qmu}\cite{FG,OddH} Let $G$ be a graph without isolated vertices. Then $q(G)\le \max\{\eta_G(u):u\in V(G)\}$.
\end{lemma}

For a connected graph $G$, $Q(G)$ is irreducible, and hence by Perron-Frobenius theorem, it has a unique unit positive eigenvector $\mathbf{x}$ corresponding to $q(G)$, called the Perron vector of $G$.

\begin{lemma}\label{perron}\cite{CRS}
Let $G$ be a connected graph with $u,v,w\in V(G)$ such that $uw\notin E(G)$ and $vw\in E(G)$. If the Perron vector $\mathbf{x}$ satisfies that $x_u\ge x_v$, then $q(G-vw+uw)>q(G)$.
\end{lemma}

\begin{lemma}\label{edgemove2}
Let $G$ be a graph with vertex $u$. Let $v\in N(u)$ with $d(v)\le d(u)-2$ and $w\in V(G)\setminus N[u]$ be a neighbor of $v$.
If $|N(v)\setminus N[u]|=1,2$, and  $N(z)\setminus \{v\}\subseteq N(v)\setminus N(u)$ for any $z\in N(v)\setminus N[u]$, then $q(G-vw+uw)>q(G)$.
\end{lemma}
\begin{proof}
Let $q=q(G)$ and $\mathbf{x}$ be the Perron vector of $G$. By Lemma \ref{perron}, it suffices to show that $x_u\ge x_v$.

By Lemma \ref{Delta}, we have $q\ge \Delta(G)+1>d(v)-1$. 

Suppose first that $|N(v)\setminus N[u]|=1$, i.e., $N(v)\setminus N[u]=\{w\}$.  As  
$N(w)\setminus \{v\}\subseteq \{w\}$, we have $d(w)=1$.  By $q\mathbf{x}=Q(G)\mathbf{x}$ at $u$ and $w$, we have
\begin{align*}
qx_u &=d(u)x_u+x_v+\sum_{z\in N(u)\setminus \{v\}}x_z,\\
qx_w &=x_w+x_v.
\end{align*}
Then 
\[
(q-1)(x_u-x_{w})=(d(u)-1)x_u+\sum_{z\in N(u)\setminus N(w)}x_z>0,
\]
so $x_u>x_w$. 
By $q\mathbf{x}=Q(G)\mathbf{x}$ at $u$ and $v$, we have 
\[
q(x_u-x_v) =d(u)x_u+x_v+\sum_{z\in N(u)\setminus \{v\}}x_z-\left(d(v)x_v+x_u+x_w+\sum_{z\in N(v)\setminus \{u,w\}}x_z\right),
\]
so 
\[
(q-d(v)+1)(x_u-x_v)=(d(u)-d(v))x_u+\sum_{z\in N(u)\setminus N[v]}x_z-x_{w}\ge x_u-x_w>0
\]
implying that $x_u\ge x_v$.

Suppose next that $|N(v)\setminus N[u]|=2$ with $N(v)\setminus N[u]=\{w,w'\}$ and $ww'\notin E(G)$. 
As  
$N(w)\setminus \{v\}\subseteq \{w,w'\}$, we have $d(w)=d(w')=1$. 
As above, we have   $x_u>x_{w}$, $x_u>x_{w'}$ and 
\begin{align*}
(q-d(v)+1)(x_u-x_v)&=(d(u)-d(v))x_u+\sum_{z\in N(u)\setminus N[v]}x_z-x_{w}-x_{w'}\\
& \ge 2x_u-x_w-x_{w'}>0,
\end{align*}
so $x_u\ge x_v$.

Now suppose next that $|N(v)\setminus N[u]|=2$ with $N(v)\setminus N[u]=\{w,w'\}$ and $ww'\in E(G)$. Then $d(w)=d(w')=2$. By $q\mathbf{x}=Q(G)\mathbf{x}$ at $u$, $w$ and $w'$, we have $x_w=x_{w'}$ and 
\[
(q-3)(x_u-x_{w})=(d(u)-3)x_u+\sum_{z\in N(u)\setminus N(w)}x_z>0,
\]
so $x_u>x_{w}$. By $q\mathbf{x}=Q(G)\mathbf{x}$ at $u$ and $v$, we have
\[
(q-d(v)+1)(x_u-x_v)=(d(u)-d(v))x_u+\sum_{z\in N(u)\setminus N[v]}x_z-2x_{w}\ge 0,
\]
so  $x_u\ge x_v$. 
\end{proof}

\begin{lemma}\label{edgemove3}
Let $u$ be a vertex of a graph $G$. Let $w_1,w_2\in V(G)\setminus N[u]$ with $N(w_2)\setminus \{w_1\}\subseteq N(u)$.
If $N(w_1)=\{w_2\}$ and $d(u)\ge d(w_2)+1$, then $q(G-w_1w_2+uw_1)>q(G)$.
\end{lemma}
\begin{proof}
Let $q=q(G)$ and $\mathbf{x}$ be the Perron vector of $G$. By Lemma \ref{perron}, it suffices to show that $x_u\ge x_{w_2}$.

By $q\mathbf{x}=Q(G)\mathbf{x}$ at $w_1$, we have
\[
(q-1)x_{w_1}=x_{w_2}.
\]
Again, by $q\mathbf{x}=Q(G)\mathbf{x}$ at $u$ and $w_2$, we have
\[
(q-d(w_2))(x_u-x_{w_2})=(d(u)-d(w_2))x_u+\sum_{z\in N(u)\setminus N(w_2)}x_z-x_{w_1}.
\]
Then
\[
\left(q-d(w_2)-\frac{1}{q-1}\right)(x_u-x_{w_2})=\left(d(u)-d(w_2)-\frac{1}{q-1}\right)x_u+\sum_{z\in N(u)\setminus N(w_2)}x_z.
\]
By Lemma \ref{Delta}, we have $q\ge \Delta(G)+1>d(w_2)+\frac{1}{q-1}$. So $x_u\ge x_{w_2}$.
\end{proof}

\begin{lemma}\label{edgemove}
Let $G$ be a graph with vertex $u$ such that $G[N(u)]$ consists of paths and $d(u)\ge 5$. Let $w\in V(G)\setminus N[u]$ with $N(w)=\{v_1,v_2\}$.  If $v_1$ and $v_2$ are two adjacent vertices in $N(u)$ with $N(v_i)\setminus (N[u]\cup \{w\})$=$\emptyset$ for $i=1,2$, then $q(G-v_1v_2+uw)>q(G)$.
\end{lemma}
\begin{proof}
Let $q=q(G)$ and $\mathbf{x}$ the Perron vector of $G$. 
If $x_{u}+x_w>x_{v_1}+x_{v_2}$, then we have by Rayleigh's principle that \[
q(G-v_1v_2+uw)-q\ge \mathbf{x}^\top (Q(G-v_1v_2+uw)-Q(G))\mathbf{x}=(x_u+x_w)^2-(x_{v_1}+x_{v_2})^2>0,
\]
as desired. So it suffices to show that $x_{u}+x_w>x_{v_1}+x_{v_2}$.

Suppose first that  $v_1$ and $v_2$ are pendant vertices in $G[N(u)]$, then $P_2=v_1v_2$ is a component of $G[N(u)]$.
In this case, $d(v_1)=d(v_2)=3$ and $x_{v_1}=x_{v_2}$. By $q\mathbf{x}=Q(G)\mathbf{x}$ at $v_i$ for $i=1,2$, we have  $qx_{v_i}=3x_{v_i}+x_u+x_{v_i}$, i.e., $(q-4)x_{v_i}=x_u+x_w$, so
\[
x_{u}+x_w-(x_{v_1}+x_{v_2})=x_{u}+x_w-2x_{v_1}=(q-6)x_{v_1}.
\]
By Lemma \ref{Delta} and the fact that $G\ncong K_{1,n-1}$, we have $q> \Delta(G)+1\ge d(u)+1\ge 6$. So $x_u+x_w>x_{v_1}+x_{v_2}$.

Suppose next that at most one of $v_1$ and $v_2$ is a pendant vertex in $G[N(u)]$.
By $q\mathbf{x}=Q(G)\mathbf{x}$ at $u$, we have
$(q-d(u))x_u=\sum_{v\in N(u)}x_v$.
By $q\mathbf{x}=Q(G)\mathbf{x}$ at $v_1$ and $v_2$,
if
one of $v_1$ and $v_2$ is a pendant vertex in $G[N(u)]$, say $v_1$, then \[
q(x_{v_1}+x_{v_2})=2x_u+4(x_{v_1}+x_{v_2})+x_{v_2}+x_{v_2'}+2x_w,
\]
where $v_2'$ is the neighbor of $v_2$ different from $v_1$ in $N(u)$,
and otherwise (neither $v_1$ nor $v_2$ is a pendant vertex in $G[N(u)]$),
 \[
q(x_{v_1}+x_{v_2})=2x_u+5(x_{v_1}+x_{v_2})+x_{v_1'}+x_{v_2'}+2x_w,
\]
where $v_i'$ is the neighbor of $v_i$ different from $v_{3-i}$ in $N(u)$ for $i=1,2$.
In either case, we have \[
(q-4)(x_{v_1}+x_{v_2})<2(x_u+x_w),
\]
i.e., \[
2((x_u+x_w)-(x_{v_1}+x_{v_2}))>(q-6)(x_{v_1}+x_{v_2}).
\]
As $q>6$, we have $x_{u}+x_w>x_{v_1}+x_{v_2}$.
\end{proof}

Given a graph $H$ with  $u\in V(H)$ and positive integers $t$ and $s$, let $H(u;t,s)$ denote the graph obtained from $H$ and $K_1\vee (P_t\cup P_s)$ by identifying $u$ and the vertex of degree $t+s$ of $K_1\vee (P_t\cup P_s)$. Let  $H(u;t,0)$ denote the graph obtained from $H$ and $K_1\vee P_t$ by identifying $u$ and a  vertex of degree $t$ of $K_1\vee P_t$.

\begin{lemma}\label{edgeshift}
If $t\ge s\ge 1$, then $q(H(u;t+1,s-1))>q(H(u;t,s))$.
\end{lemma}
\begin{proof} If $s=1$, it follows from Lemma \ref{addedges}. Suppose that $s\ge 2$.
Let $G=H(u;t,s)$, $q=q(G)$ and $\mathbf{x}$ be the Perron vector of $G$ corresponding to $q$.
Labeled the vertices of $P_t$ and $P_s$ as $u_1,\dots,u_t$ and $v_1,\dots,v_s$ in order, respectively. Let $a_i=x_{u_i}$ and $b_j=x_{v_j}$ for $i=1,\dots,t$ and $j=1,\dots,s$.

Suppose first that $a_i\ge b_{i+1}$ for all $i=1,\dots,s-1$, or $a_i\le b_{i+1}$ for all $i=1,\dots,s-1$. In the former case, as $a_1\ge b_2$ and $G-v_1v_2+u_1v_1\cong H(u;t+1,s-1)$, we have by Lemma \ref{perron} that $q(H(u;t+1,s-1))> q$.
In the latter case, as $a_{s-1}\le b_s$ and $G-u_{s-1}u_s+u_sv_s\cong H(u;t+1,s-1)$, we have by Lemma \ref{perron} again that $q(H(u;t+1,s-1))> q$.

Suppose next that there is some $i=1,\dots,s-2$ such that $a_i\ge b_{i+1}$ and $a_{i+1}< b_{i+2}$, or $a_i\le b_{i+1}$ and $a_{i+1}> b_{i+2}$. As $G'=G-u_iu_{i+1}-v_{i+1}v_{i+2}+u_iv_{i+2}+u_{i+1}v_{i+1}\cong H(u;t+1,s-1)$, we have by Rayleigh's principle that
\begin{align*}
q(G')-q&\ge\mathbf{x}^\top (Q(G')-Q(G))\mathbf{x}\\
&=(a_i+b_{i+2})^2+(a_{i+1}+b_{i+1})^2-(a_i+a_{i+1})^2-(b_{i+1}+b_{i+2})^2\\
&=2(b_{i+2}-a_{i+1})(a_i-b_{i+1})\ge 0.
\end{align*}
If $q(G')=q$, then $\mathbf{x}$ is also a Perron vector of $G'$. By $q\mathbf{x}=Q(G)\mathbf{x}$ and $q(G')\mathbf{x}=Q(G')\mathbf{x}$ at $u_i$, we have 
$qa_i=3a_i+a_{i-1}+a_{i+1}+x_u$
and $qa_i=3a_i+a_{i-1}+b_{i+2}+x_u$, implying that $a_{i+1}=b_{i+2}$, a contradiction. 
So $q(H(u;t+1,s-1))>q$.
\end{proof}

\section{Proof of Theorem \ref{str}}

Given an outerplanar graph $F$, let $\mathcal{G}_{n,F}$ denote the set of all connected $F$-free outerplanar graphs of order $n$. Theorem \ref{str} states that 
for $F=C_\ell$ with $3\le \ell\le n$ or $F=tP_k$ with $k\ge 2$ and $4\le tk\le n-1$, if $G$ maximizes the $Q$-index  among all graphs in $\mathcal{G}_{n,F}$, then $G$ contains a vertex $u$ of degree $n-1$ and $G[N(u)]$ consists of paths.
We start with the special cases when $F$ is one of  $P_4$, $2K_2$ or $C_3$. Note that $K_{1,3}$ is the unique connected graph of order $4$ that is $P_4$-free, $2K_2$-free, and  $C_3$-free. So we assume that $n\ge 5$.

\begin{theorem}\label{small}
Let $F=P_4, 2K_2, C_3$. If $G$ maximizes the $Q$-index among all graphs in $\mathcal{G}_{n,F}$ with $n\ge 5$, then $G\cong K_{1,n-1}$.  
\end{theorem}
\begin{proof}
Suppose first that $F=P_4, 2K_2$ and $G\ncong K_{1,n-1}$. Let $v$ be the vertex of maximum degree in $G$. Then $2\le d(v)\le n-2$ and there is a vertex $z\in V(G)\setminus N[v]$ adjacent to some vertex in $N(v)$, and hence $G$ contains $F$ as a subgraph, a contradiction. So $G\cong K_{1,n-1}$. Suppose in the following that $F=C_3$.
  
As $K_{1,n-1}\in \mathcal{G}_{n,F}$, we have $q(G)\ge q(K_{1,n-1})=n$.  As $G$ is $F$-free, $N(u)$ is an independent set of $G$ for any vertex $u$. It suffices to show that $\Delta(G)=n-1$. Suppose to the contrary that $\Delta(G)\le n-2$. To derive a contradiction by Lemma \ref{qmu}, it suffices to show that $\eta(u)<n$ for any $u\in V(G)$.

If $d(u)=1$, then $\eta(u)=1+d(v)<n$. 

If $d(u)=2$ with $N(u)=\{v_1,v_2\}$, then by Lemma \ref{obv} (ii), we have $d(v_1)+d(v_2)=|N(v_1)\cup N(v_2)|+|N(v_1)\cap N(v_2)|\le n-2+2=n$ and hence $\eta(u)=2+\frac{n}{2}<n$.  

If $d(u)=3$ with $n=5$ or $n=6$, then by Lemma \ref{obv} (ii), we have $\sum_{v\in N(u)}d_v=e(G)\le 5$ in the former case, and $\sum_{v\in N(u)}d_v=e(G)\le 8$ in the latter case, which follows that $\eta(u)=3+\frac{5}{3}<5$ in the former case $\eta(u)=3+\frac{8}{3}<6$ in the latter case.

If $d(u)=n-2$, then by Lemma \ref{obv} (ii), $e(G)\le n$ and so $\eta(u)\le n-2+\frac{n}{n-2}<n$.

We are left with the cases when $4\le d(u)\le n-3$ or when $d(u)=3$ with $n\ge 7$. Then $n\ge 7$ in both cases. By Lemma \ref{obv}, we have $\sum_{v\in N(u)}d(v)\le e(G)\le 2n-3$, so \[
\eta(u)\le d(u)+\frac{2n-3}{d(u)}.
\]
As the function $f(x)=x+\frac{2n-3}{x}$ is decreasing in $[0,\sqrt{2n-3}]$ and increasing in $[\sqrt{2n-3},+\infty]$, we have $\eta(u)\le \max\{f(3),f(n-3)\}<n$. 
\end{proof}

In the following, we consider $F$ as disjoint union of paths of order at least five, or cycles of order at least four. Let $\mathcal{F}$ denote the family of such these graphs, i.e., \[
\mathcal{F}=\{tP_\ell:5\le t\ell\le n-1,t\ge 1,\ell\ge 2\}\cup \{C_\ell:\ell\ge 4\}.
\]
In order to prove Theorem \ref{str}, based on Theorem \ref{small} and Lemma \ref{obv} (i), we only need to prove the following structural theorem for extremal graphs in $\mathcal{G}_{n,F}$ with $F\in \mathcal{F}$.

\begin{theorem}\label{x}
Let $F\in \mathcal{F}$. If $G$ maximizes the Q-index among all graphs in $\mathcal{G}_{n,F}$ with $n\ge 9$, then $G$ contains $K_{1,n-1}$.
\end{theorem}

\begin{proof}
As $K_1\vee (K_2\cup (n-3)K_1)\in \mathcal{G}_{n,F}$, we have by Lemma \ref{Delta} that $q(G)\ge q(K_1\vee (K_2\cup (n-3)K_1))>n$.

For $u\in V(G)$, let $W=V(G)\backslash N[u]$.

\begin{claim}\label{BC} $\eta(u)\le n$ if $d(u)\le n-4$.
\end{claim}
\begin{proof}
From Lemma \ref{obv} (i) and (ii), we have $e(N(u))\le d(u)-1$ and $e(N(u),W)\le 2|W|$.

If $3\le d(u)\le n-4$, we have
\begin{align*}
\sum_{v\in N(u)}d(v)&=d(u)+2e(N(u))+e(N(u),W)\\
&\le d(u)+2(d(u)-1)+2|W|\\
&=2n+d(u)-4.
\end{align*}
Then
\[
\eta(u)\le d(u)+1+\frac{2n-4}{d(u)}.
\]
As $h(x)=x+1+\frac{2n-4}{x}$ is decreasing in $[0,\sqrt{2n-4}]$ and increasing in $[\sqrt{2n-4},+\infty]$, we have $\eta(u)\le \max\{h(3),h(n-4)\}< n$.

If $d(u)=2$, then by Lemma \ref{obv} (ii), we have $d(v_1)+d(v_2)=|N(v_1)\cup N(v_2)|+|N(v_1)\cap N(v_2)|\le n+2$ and hence $\eta(u)=2+\frac{n+2}{2}<n$. 

If $d(u)=1$, we have $\eta(u)\le 1+\Delta(G)\le 1+n-1=n$.
\end{proof}
 
If $\Delta (G)\le n-4$, then we have by Lemma \ref{qmu} and Claim \ref{BC} that $q(G)\le n$, a contradiction. So $\Delta(G)\ge n-3$. It suffices to show that $\Delta(G)=n-1$.  Suppose to the contrary that $\Delta(G)\le n-2$. Then  $\Delta(G)=n-3, n-2$.

\noindent
{\bf Case 1.} $\Delta(G)=n-3$.

In view of Claim \ref{BC}, we assume that $u$ is  a vertex of $G$ with maximum degree.
Let $w_1$ and $w_2$ denote the two vertices in $V(G)\setminus N[u]$, i.e., $W=\{w_1,w_2\}$.

\begin{claim}\label{n-3}
$G[N(u)]\cong P_{n-3}$.
\end{claim}
\begin{proof}
By Lemma \ref{obv} (i) and (ii), each vertex $v\in V(G)\setminus \{u\}$ has at most two neighbors in $N(u)$, so $d(v)\le 5$, implying that $\eta(v)\le n$.

By Lemma \ref{obv} (ii), $e(N(u),W)\le 4$, implying that 
\[
e(W)+e(N(u),W)\le 5.
\] 
Suppose that $e(W)+e(N(u),W)=5$. Then $e(N(u),W)= 4$ and $w_1w_2\in E(G)$. By Lemma \ref{obv} (iv), $N_{N(u)}(w_1)\cap N_{N(u)}(w_2)=\emptyset$. So $G[\{u\}\cup N[w_1]\cup N[w_2]]$ contains a $K_{2,3}$-minor, a contradicting the fact that $G$ is outerplanar. It follows that $e(W)+e(N(u),W)\le 4$. Thus
\[
e(G)\le d(u)+e(N(u))+4=e(N(u))+n+1
\]
By Lemma \ref{obv} (i), $G[N(u)]$ consists of paths, so $e(N(u))\le n-4$ with equality if and only if $G[N(u)]$ is a path.  If $e(N(u))\le n-5$, then $\eta(u)\le n-3+\frac{1}{n-3}(2e(N(u))+n+1)\le n$, so we have by Claim \ref{BC} and Lemma \ref{qmu} that $q(G)\le n$, a contradiction. Thus $e(N(u))=n-4$ and $G[N(u)]\cong P_{n-3}$.
\end{proof}

By Claim \ref{n-3}, $G[N[u]]$ contains $C_\ell$ for $\ell=3,\dots,n-2$. Moreover, as $G$ is connected, at least one of $w_1,w_2$ is adjacent to some vertex of $N(u)$, implying that $G$ contains $tP_\ell$ with $t\ell\le n-1$. So it remains to show the case when $F=C_{n-1}$ or $C_n$. 

By Lemma \ref{obv} (ii), $d_{N(u)}(w_i)\le 2$ for $i=1,2$. Assume that $d_{N(u)}(w_1)\le d_{N(u)}(w_2)$. As $G$ is connected, $d_{N(u)}(w_2)\ge 1$.

Suppose first that $d_{N(u)}(w_2)=1$. Then $d_{N(u)}(w_1)\le 1$. 
Let 
\[G'=G-E(N(u),W)+uw_1+uw_2.\] 
It is obvious that $G'$ is contained in $K_1\vee (P_{n-3}\cup K_2)$, which does not contain $F$ as a subgraph. Let $v$ denote the neighbor of $w_2$ in $N(u)$. If $d_{N(u)}(w_1)=0$, then we have by Lemmas \ref{addedges} and \ref{edgemove2} that $q(G')>q(G-vw_2+uw_2)>q(G)$, a contradiction. This shows that $d_{N(u)}(w_1)=1$. Let $v'$ denote the neighbor of $w_1$ in $N(u)$. 
If $w_1w_2\notin E(G)$, then we have by Lemma \ref{addedges} and Lemma \ref{edgemove2} twice that $q(G')>q(G-vw_2+uw_2-v'w_1+uw_1)>q(G)$, also a contradiction. So $w_1w_2\in E(G)$. It then follows that either $v'=v$ or $vv'\in E(G)$, as otherwise, $G[\{u\}\cup W\cup V^*]$ contains a $K_4$-minor, where $V^*$ is the vertex set of path between $v$ and $v'$ in $G[N(u)]$, a contradiction. If $vv'\in E(G)$, then $G$ contains $C_{n-1}$ and $C_n$, contradicting that $G$ is $F$-free with $F=C_{n-1}$ or $C_n$. Thus $v'=v$. By Lemma \ref{edgemove2}, $q(G')>q(G)$, a contradiction. 

Suppose next that $d_{N(u)}(w_2)=2$. As $G[N(u)]\cong P_{n-3}$, we have by Lemma \ref{obv} (iii) that the two neighbors of $w_2$ in $N(u)$ are adjacent, implying that $G-w_1$ contains $C_{n-1}$.
Moreover, if $d_{N(u)}(w_1)=2$, then we have by by Lemma \ref{obv} (iii) that the two neighbors of $w_1$ are adjacent and by Lemma \ref{obv} (iv) that $N_{N(u)}(w_1)\cap N_{N(u)}(w_2)=\emptyset$, so $G$ contains $C_n$, a contradiction. Thus  $d_{N(u)}(w_1)\le 1$ and $F=C_n$.
As $G$ is connected, we have $1\le d(w_1)\le 2$. 
If $d(w_1)=1$ and its unique neighbor is denoted by $v_1$, then $G-v_1w_1+uw_1$ is also $F$-free, so from Lemma \ref{edgemove2} or \ref{edgemove3}, we have $q(G-v_1w_1+uw_1)>q(G)$, a contradiction. Thus $d(w_1)=2$, i.e., $w_1w_2\in E(G)$ and $w_1$ has a neighbor $v_1'$ in $N(u)$. Let $v_2$ and $v_2'$ denote the two neighbors of $w_2$ in $N(u)$. By Lemma \ref{obv}, $v_2v_2'\in E(G)$.   If $v_1'\notin \{v_2,v_2'\}$, then $G[\{u,v_1',v_2,v_2',w_1,w_2\}]$ contains $K_{2,3}$ as a minor, a contradiction. So $v_1'=v_2$ or $v_2'$, implying that $G$ contains $C_n$,  also a contradiction. 

\noindent
{\bf Case 2.} $\Delta(G)=n-2$.

Let $u$ be a vertex of maximum degree of $G$. 
Let $w$ denote the vertex that is not adjacent to $u$ in $G$. By Lemma \ref{obv} (ii), $d(w)\le 2$.

Suppose that $d(w)=1$. Let $v$ be the  neighbor of $w$. Then $v\in N(u)$. Let $G'=G-vw+uw$. 
We claim that $G'\in \mathcal{G}_{n,F}$. Otherwise, $G'$ contains $F$, then $F$ contains  $uw$, which implies that $F\cong tP_\ell$ for some $t\ge 1$ and $\ell\ge 2$ with $t\ell\le n-1$. Let $z$ be a vertex of $G$ not contained in $F$. Note that $uz\in E(G)$ and $G'[E(F)\setminus \{uw\}\cup \{uz\}]\cong F$ is a subgraph of $G$, contradicting the fact that $G$ is $F$-free. This shows that $G'\in \mathcal{G}_{n,F}$. 
It then follows from Lemma \ref{edgemove2} that $q(G')>q(G)$, a contradiction. 

Suppose now that $d(w)=2$.  The neighbors of $w$ are denoted by $v_1$ and $v_2$.

\noindent
{\bf Case 2.1.} $v_1v_2\in E(G)$. 

Let $\widehat{G}=G-v_1v_2+uw$. By Lemma \ref{edgemove}, we have $q(\widehat{G})>q(G)$. If  $\widehat{G}\in \mathcal{G}_{n,F}$, then we have a contradiction. Suppose that this is not true. Then $\widehat{G}$ contains $F$, so  $F$ must contain  $uw$.  

Suppose first that $F$ is a cycle. As $d_{F}(w)=2$, $F$ only contains exactly one of $v_1,v_2$, say $v_2$. Let $z$ be the neighbor of $u$ in $F$ different from $w$, and $z'$ be the neighbor of $z$ in $F$ different from $u$. Let $F'=F-uz-zz'-uw+uz'+wv_1+uv_1$. Then $F'\cong F$ and $G$ contains $F'$, a contradiction. So $\widehat{G}\in \mathcal{G}_{n,F}$.

Suppose next that $F=tP_\ell$ for some $t\ge 1$ and $\ell\ge 2$ with $5\le t\ell\le n-1$. If $w$ is a pendant vertex in $F$, then, with $z$ being a neighbor of $u$ not contained in $F$, $\widehat{G}[E(F)\setminus \{uw\}\cup \{uz\}]\cong F$ is a subgraph of $G$, contradicting the fact that $G$ is $F$-free. So $w$ is not a pendant vertex, implying that either $wv_1\in E(F)$ or $wv_2\in E(F)$, say $wv_2\in E(F)$. Let $F_1$ be the component of $F$ containing $uw$ and $w_1$ be the pendant vertex of $F_1$ with shorter distance to $v_2$ than that to $w$. Then $F-uw+uw_1\cong F$ is a subgraph of $G$, also a contradiction. So $\widehat{G}\in \mathcal{G}_{n,F}$.

\noindent
{\bf Case 2.2.} $v_1v_2\notin E(G)$. 

By Lemma \ref{obv} (iii), $v_1$ and $v_2$ are pendant vertices and in different components in $G[N(u)]$. Let $V_i$ denote the set of vertices of the component containing $v_i$ in $G[N(u)]$ for $i=1,2$. 
Let $G^*=G-v_1w+uw$. By Lemma \ref{edgemove2}, we have $q(G^*)>q(G)$. To derive a contradiction, we need to show that $G^*\in \mathcal{G}_{n,F}$. Suppose to the contrary that $G^*$ contains $F$, so $F$ must contain the edge $uw$.  

Suppose first that $F$ is a cycle. As $\{u,w\}$ is a cut set in both $G$ and $G^*$, and $v_1w\notin E(G^*)$, $F$ only contains vertices in $V_2$, except $u$ and $w$. Let $z$ be the neighbor of $u$ in $F$ different from $w$, and $z'$ be the neighbor of $z$ in $F$ different from $u$. Let $F'=F-uz-zz'-uw+uz'+wv_1+uv_1$. Then $F'\cong F$ and $G$ contains $F'$, a contradiction. So $G^*\in \mathcal{G}_{n,F}$.

Suppose next that $F=tP_\ell$ for some $t\ge 1$ and $\ell\ge 2$ with $5\le t\ell\le n-1$. If $w$ is a pendant vertex in $F$, then, with $z$ being a neighbor of $u$ not contained in $F$, $G^*[E(F)\setminus \{uw\}\cup \{uz\}]\cong F$ is a subgraph of $G$, contradicting the fact that $G$ is $F$-free. So $w$ is not a pendant vertex, i.e., $uw,v_2w\in E(F)$. Let $F_1$ denote the component of $F$ containing $w$. Let $w_1$ denote the pendant vertex of $F_1$ in $V_2$. Then $F-uw+uw_1\cong F$ is a subgraph of $G$, also a contradiction. This shows that $G^*\in \mathcal{G}_{n,F}$, as desired.
\end{proof}

Theorem \ref{str} follows from the two previous theorems.

\section{Proof of Theorems \ref{cycle} and \ref{path}}

Now, we are ready to prove Theorems \ref{cycle} and \ref{path}.

\begin{proof}[Proof of Theorem \ref{cycle}]
Suppose that $G$ maximizes the $Q$-index among all connected outerplanar $C_\ell$-free graphs of order $n$.

If $\ell=3$, then by Theorem \ref{small}, $G\cong K_{1,n-1}=K_1\vee (n-1)K_1$, as desired. 

Suppose in the following that $\ell\ge 4$. 

By Theorem \ref{x}, $\Delta(G)=n-1$. Let $u$ be the vertex with degree $n-1$.
As $G$ is $C_{\ell}$-free, $G[N(u)]$ is $P_{\ell-1}$-free. By Lemma \ref{obv} (i),
$G[N(u)]$ consists of paths of order at most $\ell-2$. Note that $r=n-1-\alpha (\ell-2)$. 
By Lemma \ref{edgeshift} repeatedly, $G[N(u)]\cong \alpha P_{\ell-2}\cup P_r$, so 
$G\cong K_1\vee(\alpha P_{\ell-2}\cup P_r)$.
\end{proof}

\begin{proof}[Proof of Theorem \ref{path}]
Suppose that $G$ maximizes the Q-index among all connected outerplanar $tP_\ell$-free graphs.

If $t=1$ and $\ell=4$ or $t=2$ and $\ell=2$, then by Theorem \ref{small}, $G\cong K_{1,n-1}=K_1\vee (n-1)K_1$, as desired.

Suppose in the following that $t\ell\ge 5$. From the lower bound of $n$, we have $t\ell \le n-1$.

By Theorem \ref{x}, $\Delta(G)=n-1$. Let $u$ be the vertex with degree $n-1$. By Lemma \ref{obv} (i), $G[N(u)]$ is a disjoint union of paths. Assume that $G[N(u)]\cong \cup_{i=1}^sP_{a_i}$ and $a_1\ge \dots \ge a_s$.
As $G$ is $tP_{\ell}$-free, $a_1+a_2\le t\ell-2$. By Lemma \ref{edgeshift}, $a_1+a_2=t\ell-2$ and $a_2=\dots =a_{s-1}$.

Let $q=q(G)$ and $\mathbf{x}$ be the Perron vector of $G$ with $x_u=1$.

\begin{claim}\label{vector}
For any $v\in N(u)$, $\frac{1}{q}<x_v<\frac{1}{q}+\frac{30}{q^2}$.
\end{claim}
\begin{proof}
By $q\mathbf{x}=Q(G)\mathbf{x}$ at $u$ and any $z\in N(u)$, we have \[
(q-d(z))(x_u-x_z)=(n-2-d(z))x_u+\sum_{w\in N(u)\setminus N(z)}x_w>0,
\]
i.e., $x_z<x_u=1$.
Note that $d(v)\le 3$.
By $q\mathbf{x}=Q(G)\mathbf{x}$ at $v$, we have
\begin{equation}\label{xv}
qx_v=d(v)x_v+x_u+\sum_{z\in N(v)\setminus\{u\}}x_z.
\end{equation}
Recall that $x_z<1$, which implies that $x_v<\frac{6}{q}$.
From Eq.~\eqref{xv} again, we have $qx_v<1+\frac{30}{q}$, implying that $x_v<\frac{1}{q}+\frac{30}{q^2}$.
On the other hand, by Eq.~\eqref{xv}, we have $x_v>\frac{1}{q}$.
\end{proof}

We'll show that $a_2=\lfloor \frac{\ell-1}{2}\rfloor$ if $t=1$, and $a_2=\ell-1$ if $t\ge 2$. Suppose to the contrary that $a_2\le \lfloor \frac{\ell-3}{2}\rfloor$ if $t=1$, and $a_2\le \ell-2$ if $t\ge 2$. In the former case, we have  \[
s=\left\lceil\frac{n-1-a_1}{a_2}\right\rceil+1=\left\lceil\frac{n-\ell+1}{a_2}\right\rceil+2\ge \left\lceil\frac{n-\ell+1}{\lfloor\frac{\ell-3}{2}\rfloor}\right\rceil+2\ge \left\lfloor\frac{\ell-3}{2}\right\rfloor+3\ge a_2+3.
\]
In the latter case, we have \[
s=\left\lceil \frac{n-1-a_1}{a_2} \right\rceil+1=\left\lceil \frac{n-t\ell+1}{a_2} \right\rceil+2\ge\frac{n-t\ell+1}{\ell-2}+2\ge \ell+1 \ge a_2+3.
\]
This shows that $s\ge a_2+3$ in either case.
Let $v_i$ denote an end vertex of the path $P_{a_i}$ for $i=1,\dots,a_2+2$ and $v_1'$ denote the neighbor of $v_1$ in $P_{a_1}$. Let $P_{a_{s-1}}=w_1\dots w_{a_2}$ and
\[
G'=G-v_1v_1'-\sum_{i=1}^{a_2-1}w_iw_{i+1}+v_1v_2+\sum_{i=3}^{a_2+2}v_iw_{i-2}.
\]
This implies that $G'$ is obtained from $G$ by deleting $a_2$ edges and adding $a_2+1$ edges.
By Claim \ref{vector}, we have
\[
\mathbf{x}^\top (Q(G')-Q(G))\mathbf{x}\ge 4(a_2+1)\frac{1}{q^2}-4a_2\left(\frac{1}{q}+\frac{30}{q^2}\right)^2=\frac{4}{q^4}\left(q^2-60a_2q-900a_2\right).
\]
As $n\ge 30\lfloor\frac{\ell-3}{2} \rfloor+30\sqrt{\lfloor\frac{\ell-3}{2} \rfloor \big (\lfloor\frac{\ell-3}{2} \rfloor+1\big )}$ if $t=1$ and $n\ge 30(\ell-2)+30\sqrt{(\ell-2)(\ell-1)}$ if $t\ge 2$, we have $n\ge 30a_2+30\sqrt{a_2(a_2+1)}$ and so $q>n\ge 30a_2+30\sqrt{a_2(a_2+1)}$, implying that $\mathbf{x}^\top (Q(G')-Q(G))\mathbf{x}>0$. So by Rayleigh's principle, $q(G')>q$, a contradiction. Thus $a_2=\lfloor \frac{\ell-1}{2}\rfloor$ if $t=1$, and $a_2=\ell -1$ if $t\ge 2$.

Therefore, $a_1=\lceil\frac{\ell-1}{2} \rceil$ and $G\cong K_1\vee \big(P_{\lceil\frac{\ell-1}{2} \rceil}\cup \alpha P_{\lfloor\frac{\ell-1}{2} \rfloor}\cup P_r\big)$ if $t=1$, and $a_1=t\ell-\ell-1$ and $G\cong K_1\vee (P_{t\ell-\ell-1}\cup \alpha P_{\ell-1}\cup P_r)$ if $t\ge 2$.
\end{proof}

\bigskip

\noindent
{\bf Acknowledgement} \ The research is supported by the National Natural Science Foundation of China (No.~12571364).

\bigskip

\noindent
{\bf Conflict of Interest} \
The authors declare that there is no conflict of interest.

\bigskip

\noindent
{\bf Data availability} \ 
No data was used for the research described in the article.

\end{document}